\documentclass{amsart}
\usepackage{amsfonts}
\usepackage{amsmath,amssymb}
\usepackage{amsthm}
\usepackage{amscd}
\usepackage{graphics}
\usepackage{graphicx}

\theoremstyle{remark}{
\newtheorem{Def}{{\rm Definition}}

}
\theoremstyle{plain}
{

\newtheorem{Thm}{Theorem}
\newtheorem{MainThm}{Main Theorem}

}

\begin{document}
\title[Regions surrounded by parabolas and trees representing their shapes canonically]{Regions surrounded by parabolas in the plane and trees representing their shapes respecting their natural projection to the line}
\author{Naoki kitazawa}

\address{Institute of Mathematics for Industry, Kyushu University, 744 Motooka, Nishi-ku Fukuoka 819-0395, Japan\\
 TEL (Office): +81-92-802-4402 \\
 FAX (Office): +81-92-802-4405 \\
}
\email{naokikitazawa.formath@gmail.com}
\urladdr{https://naokikitazawa.github.io/NaokiKitazawa.html}
\maketitle
\begin{abstract}
The author has been interested in regions surrounded by real algebraic curves of degree $1$ or $2$ in the plane. The author is mainly interested in their shapes and combinatorics. This is a fundamental and natural problem in mathematics being also elementary and connected to various fields. The shapes are understood via graphs the regions collapsing to respecting the canonical projection onto the 1st component.

Our main result is the following: each tree is realized by regions surrounded by parabolas of two types, here.

Related studies are elementary and interesting and surprisingly, this explicit field is started very recently, by Bodin, Popescu-Pampu and Sorea in the 2020s. After that, this is developing, due to the author. The author also investigates this motivated by studies on explicit construction of real algebraic maps onto the regions locally so-called {\it moment maps}: this comes from singularity theory of differentiable maps and real algebraic geometry. 
\end{abstract}
%【REVISE】 combinatoric ～ is → combinatorial object. It is .
%【REVISE】  such that a point is a vertex if and only if the corresponding connected component of the level set contains some singular points → whose vertex set is the set of all points containing some singular points in the corresponding connected component of the level set .
%【REVISE】 We delete "extending the result before".
\section{Introduction.}
\label{sec:1}
Regions in the $k$-dimensional Euclidean space ({\it real affine space}) ${\mathbb{R}}^k$ (with $\mathbb{R}:={\mathbb{R}}^1$) surrounded by real algebraic hypersurfaces there are of fundamental mathematical objects. We are concerned with regions surrounded by real algebraic curves of degree $1$ or $2$. The author is mainly interested in their shapes and combinatorics.

We understand them by graphs the regions collapse to respecting the canonical projection ${\pi}_{2,1}:{\mathbb{R}}^2 \rightarrow \mathbb{R}$, where ${\pi}_{k_1,k_2}:{\mathbb{R}}^{k_1} \rightarrow {\mathbb{R}}^{k_2}$ denotes the projection ${\pi}_{k_1+k_2,k_2}(x)=x_1 \in {\mathbb{R}}^{k_1}$ ($k_1,k_2>0$ and $x:=(x_1,x_2) \in {\mathbb{R}}^{k_1} \times {\mathbb{R}}^{k_2}$). In short, we consider the graph represented as the set of all level sets of the restriction of the projection to the closure of the region taken in ${\mathbb{R}}^2$ and topologized with the quotient topology of the closure canonically.

Our main theorem is as follows. A {{\it tree} is a finite graph with no cycle.
\setcounter{MainThm}{-1}
\begin{MainThm}
\label{mthm:0}
Each tree is realized by regions surrounded by parabolas of two types, here.
\end{MainThm}

The graph presented above is first presented in \cite{bodinpopescupampusorea} (see also \cite{sorea1, sorea2}). The graph is named the {\it Poincar\'e-Reeb graph} and it is essentially same as the {\it Reeb graph} of a smooth function on a manifold with no boundary. Reeb graphs of smooth functions are classical and fundamental tools and have been strong tools in understanding the manifolds via nice smooth functions such as Morse functions (\cite{reeb}). 
Bodin, Popescu-Pampu, and Sorea consider a finite graph embedded in the plane generic way and prove the existence of a region surrounded by mutually disjoint ({\it non-singular}) real algebraic curves. Approximation respecting derivatives is a main ingredient in the argument. The author was fascinated with this study and have been studying explicit realization of regions surrounded by real algebraic hypersurfaces of certain explicit classes and yielding the natural graphs as above satisfying suitable explicit conditions. This study also comes from singularity theory of differentiable maps, where the author has been interested in constructing explicit nice smooth or real algebraic (analytic) functions and maps. Especially, the author has been interested in constructing explicit functions and maps onto the closures of the regions of certain classes generalizing the canonical projection of the $m$-dimensional unit sphere $S^{m}:=\{x=(x_1,\cdots, x_{m+1}) \in {\mathbb{R}}^{m+1} \mid {\Sigma}_{j=1}^{m+1} {x_j}^2=1\}$ onto the $n$-dimensional unit disk $D^{n}:=\{x=(x_1,\cdots x_n) \in {\mathbb{R}}^n \mid {\Sigma}_{j=1}^{n} {x_j}^2 \leq 1\}$ (the restriction of ${\pi}_{m+1,n}$ to $S^m$). Note that the maps onto the closure of the regions are locally so-called {\it moment} maps. For moment maps, see \cite{buchstaberpanov, delzant} for example. Real algebraic functions here are obtained by composing ${\pi}_{2,1}$. For this, \cite{kitazawa3} is a pioneering study, remarked in \cite{kitazawa9}. Preprints \cite{kitazawa5, kitazawa6, kitazawa7, kitazawa8, kitazawa13, kitazawa14, kitazawa15, kitazawa18, kitazawa20} are closely related to this. Studies of regions are presented by the author, motivated by previous studies such as \cite{kitazawa3, kitazawa5, kitazawa7, kitazawa8}, and a related preprint has been presented in  \cite{kitazawa10} first, by the author. \cite{kitazawa11, kitazawa12, kitazawa16, kitazawa17, kitazawa19} follow this well, for example. Related to this, realization of finite graphs as Reeb graphs of nice differentiable (smooth) functions was first studied in \cite{sharko}. There, a smooth function $c:X \rightarrow \mathbb{R}$ on a (suitable) closed surface $X$ at each critical point $p$ of which the form is $c(z)={\Re}z^{l} +c(p)$ or $c(z)=z\bar{z}+c(p)$, is reconstructed from a given finite graph satisfying certain additional conditions: here, $z$ is a complex number, $\bar{z}$ is its conjugate, $\Re z$ is the real part of a complex number $z$, $l>1$ is a positive integer, and $p$ is identified with $z=0$ for a suitable local coordinate. This is improved in \cite{masumotosaeki} for arbitrary finite graphs, later, reconstruction of Morse functions is considered systematically first in \cite{michalak}, and the author has contributed such reconstruction of smooth functions with level sets having prescribed shapes, in \cite{kitazawa1, kitazawa2, kitazawa4}.
 
This paper is organized as follows. In the second section, we explain fundamental notions and notation.
In the third section, we formulate our regions, reviewing the preprints \cite{kitazawa17, kitazawa18, kitazawa19} of the author, for example, in a self-contained way. In the fourth section, we restate Main Theorem \ref{mthm:0} in a revised way, as Main Theorem \ref{mthm:1}. We also present Main Theorem \ref{mthm:2}, as another result. This is related to an explicit Reeb graph presented in  \cite[Theorem 1]{kitazawa20}. 
\section{Preliminaries.}
We review or introduce some fundamental notions and notation.
Let $\mathbb{Z}$ denote the set of all integers.

For a map $c:X \rightarrow Y$, the restriction of $c$ to a subset $Z \subset X$ is denoted by $c {\mid}_Z$.
 
For a topological space $X$ and its subspace $Y \subset X$, we use $\overline{Y}^{X}$ for the closure of $Y$ in $X$.

A map between topological spaces $c:X \rightarrow Y$ is said to be {\it proper} if the preimage $c^{-1}(K)$ is compact for any compact set $K \subset Y$. 

The {\it Reeb space} $R_c$ of a (continuous) map $c:X \rightarrow Y$ between topological spaces means the quotient space $R_c:=X/{\sim c}$ defined in the following way. We can easily defined the equivalence relation ${\sim}_c$ on $X$ by the rule that $p_1 {\sim}_c p_2$ holds if and only if $p_1$ and $p_2$ are in a same connected component of a same preimage $c^{-1}(q)$. In the case $Y$ is an ordered set such as $\mathbb{R}$, which is seen as one of fundamental totally ordered set, each preimage $c^{-1}(q)$ is called a {\it level set} of $c$ and each connected component of it is a {\it contour} of $c$. We can define the quotient map $q_c:X \rightarrow R_c$ with the uniquely defined continuous map $\bar{c}$ yielding the relation $c=\bar{c} \circ q_c$.  

For a differential manifold $X$, let $T_p X$ denote the tangent vector space at $p \in X$.
A {\it singular} point $p$ of a differential map $c:X \rightarrow Y$ between differentiable manifolds is a point, where the differential ${dc}_p:T_p X \rightarrow T_{c(p)} Y$ is not of the full rank. 
For a singular point $p$ of $c$, $c(p)$ is a {\it singular value} of $c$.
The map ${dc}_p$ is a linear map between the real vector spaces.

In the case $Y:=\mathbb{R}$, we also use "{\it critical}" instead of {\it singular}. A contour of a differentiable function $c:X \rightarrow \mathbb{R}$ is {\it critical} ({\it regular}) of it contains some critical points (resp. no critical point).

Let $0 \in {\mathbb{R}}^k$ denote the origin. A {\it real algebraic} manifold is a union $Z_{c}$ of some connected components of the zero set $c^{-1}(0)$ of a real polynomial map $c:{\mathbb{R}}^{k_1} \rightarrow {\mathbb{R}}^{k_2}$ such that $c {\mid}_{Z_{c}}$ has no singular point. In other words, this is assumed to be {\it non-singular}.

A {\it graph} means a $1$-dimensional connected CW complex the closure of whose $1$-cell ({\it edge}) is homeomorphic to $D^1$. A {\it vertex} of graph means a 0-cell of it.
The {\it degree} of a vertex of a graph means the number of all edges of the graph incident to it.
A {\it tree} is a finite graph whose 1st Betti number is $0$, or equivalently, which has no {\it cycle}. 

The Reeb space of a smooth function $c:X \rightarrow \mathbb{R}$ on a closed and connected manifold the number of all of whose critical values is finite is a graph by the following rule. A point $v$ there is a vertex of it if and only if ${q_c}^{-1}(v)$ is a critical contour of $c$. This is the {\it Reeb graph} of $c$.
\section{Reviewing RA-regions and their Poincar\'e-Reeb graphs.}  
\begin{Def}
\label{def:1}
Let $D \subset {\mathbb{R}}^n$ be a non-empty, connected and open set of ${\mathbb{R}}^n$. Let $\{S_{\lambda}\}_{\lambda \in \Lambda}$ be a family of real algebraic manifolds of dimension $n-1$ in ${\mathbb{R}}^n$ satisfying the following.
\begin{itemize}
\item Each $S_{\lambda}$ of which is a union of some components of the zero set of a real polynomial $f_{\lambda}$ and $f_{\lambda}$ is used for giving $S_{\lambda}$ the structure of the real algebraic manifold.
\item The intersection $S_{\lambda}  \bigcap {\overline{D}}^{{\mathbb{R}}^n}$ is non-empty for each $\lambda \in \Lambda$.
\end{itemize}
If the following hold in addition, then the pair $(D,\{S_{\lambda}\}_{\lambda \in \Lambda})$ is said to be a {\it real algebraic region} or an {\it RA-region} of ${\mathbb{R}}^n$.  
\begin{enumerate}
\item \label{def:2.1} There exists a connected open neighborhood $U_D$ of the closure ${\overline{D}}^{{\mathbb{R}}^k}$ such that $D=U_D \bigcap {\bigcap}_{\lambda \in \Lambda} \{x \mid f_{\lambda}(x)> 0\}$ and $U_D \bigcap S_{\lambda}=U_D \bigcap \{f_{\lambda}(x)=0\}$ hold.
\item \label{def:2.2} Each point of $p \in {\overline{D}}^{{\mathbb{R}}^k}$ is contained in at most $n$ distinct manifolds from $\{S_{\lambda}\}_{\lambda \in \Lambda}$. 
Let $p$ be an  arbitrary point from $({\bigcup}_{\lambda \in \Lambda} S_{\lambda}) \bigcap {\overline{D}}^{{\mathbb{R}}^k}$.
For $p \in {\bigcap}_{j=1}^{n_p} S_{{\lambda}_j} \bigcap {\overline{D}}^{{\mathbb{R}}^n}$, contained in these exactly $1 \leq n_p \leq n$ distinct manifolds from $\{S_{\lambda}\}_{\lambda \in \Lambda}$, it holds that $\dim {\bigcap}_{j=1}^{n_p} T_p S_{{\lambda}_j}=n-n_p$.
\end{enumerate}

\end{Def}

\begin{Def}
\label{def:2}
A point $p \in {\overline{D}}^{{\mathbb{R}}^n}$ is called a {\it singular} point of an RA-region $(D,\{S_{\lambda}\}_{\lambda \in \Lambda})$ if and only if either of the following two holds.
\begin{itemize}
\item \begin{itemize}
\item The point $p$ is in exactly $n_p<n$ distinct manifolds $S_{{\lambda}_j}$ ($1 \leq j \leq n_p$) here. Let $S_{{\lambda}_j,1 \leq j \leq n_p}$ denote the set of all points of ${\overline{D}}^{{\mathbb{R}}^n}$ contained in these $n_p$ manifolds $S_{{\lambda}_j}$ and not contained in any other hypersurface $S_{\lambda}$. This space is an ($n-n_p$)-dimensional manifold with no boundary. 
\item 
The point $p$ is also a critical point of the restriction ${\pi}_{n,1} {\mid}_{S_{{\lambda}_j,1 \leq j \leq n_p}}$.
\end{itemize}
\item The point $p$ is in exactly $n$ distinct manifolds $S_{{\lambda}_j}$  ($1 \leq j \leq n$).
\end{itemize}

\end{Def}

Now consider the situation that the quotient map $q_{{\pi}_{n,1} {\mid}_{{\overline{D}}^{{\mathbb{R}}^n}}}$ is proper and that the cardinality of $\Lambda$ is finite.
Due to the algebraic conditions and the finiteness, yielding the finiteness of the number of values of the function at singular points of the RA-region, and fundamental and important arguments in \cite{saeki1, saeki2}, the Reeb space has the structure of a graph by defining its vertex $v$ as a point whose preimage ${q_{{\pi}_{n,1}}}^{-1}(v)$ contains a singular point of the RA-region. For this, consult also articles on stratified Morse theory, presented in \cite{hamm, massey}, for example. 
\begin{Def}
\label{def:3}
This is the {\it Poincar\'e-Reeb graph} of the RA-region and denoted by ${\rm PR}_{(D,\{S_{\lambda}\}_{\lambda \in \Lambda})}$.
\end{Def}
We can argue in this way for some cases where such finiteness is dropped. For this, enjoy Main Theorem \ref{mthm:2}.
\section{Main Theorem \ref{mthm:0} revisited as Main Theorem \ref{mthm:1}, and Main Theorem \ref{mthm:2}.}
A {\it circle} $\{(x_1,x_2) \in {\mathbb{R}}^2 \mid {(x_1-p_1)}^2+{(x_2-p_2)}^2=r_{1,2}\}$ centered at $(p_1,p_2) \in {\mathbb{R}}^2$ and of radius $r_{1,2}>0$ is an important $1$-dimensional real algebraic manifold (real algebraic curve). In the present paper,  a {\it parabola} $\{(x_1,x_2) \in {\mathbb{R}}^2 \mid x_2=a{(x_1-p_1)}^2+p_2\}$ with $(p_1,p_2) \in {\mathbb{R}}^2$ and $a \neq 0$ is also an important real algebraic curve.
A {\it hyperbola} $\{(x_1,x_2) \in {\mathbb{R}}^2 \mid r=a{(x_1-p_1)}^2-b{(x_1-p_1)}^2\}$ with $(p_1,p_2) \in {\mathbb{R}}^2$ and $a,b,r>0$ is also an important real algebraic curve. This consists of exactly two connected components.

We also expect readers to have elementary knowledge on so-called Euclidean plane geometry.
Two subsets in a real affine space are {\it congruent} if one subset is changed to the remaining subset by a composition of finitely many rotations, parallel transformations, and reflections, where we assume elementary terminologies, notion and arguments on Euclidean geometry. We can also say a subset is {\it congruent} to the other.

A set in ${\mathbb{R}}^2$ congruent to such a real algebraic curve of degree $2$ is also called a {\it circle}, a {\it parabola}, or a {\it hyperbola}, for example.

\begin{MainThm}
\label{mthm:1}
Let $S_P$ be a parabola.
Each tree $G$ is isomorphic to the Poincar\'e-Reeb graph of some RA region surrounded by parabolas of congruent to either $S_P$ or another suitably chosen parabola. 
\end{MainThm}
\begin{proof}
We first consider a tree the degrees of whose vertices of which is not $2$ (STEP 1-1). After that, we consider a general tree (STEP 1-2). \\
\ \\
STEP 1-1 The case of a tree the degrees of whose vertices of which are not $2$. \\
STEP 1-1-1 The case where the tree is homeomorphic to $D^1$. \\
We choose a parabola $\{(x_1,x_2) \in {\mathbb{R}}^2 \mid x_2=a{(x_1-p_1)}^2+p_2\}$ with $(p_1,p_2) \in {\mathbb{R}}^2$ and $a>0$, and another parabola $\{(x_1,x_2) \in {\mathbb{R}}^2 \mid x_2=-a{(x_1-p_1)}^2+p_2\}$, congruent to the previous parabola. By considering the projection ${\pi}_{2,1}$, we can see that this is a desired case. \\
\ \\
STEP 1-1-2 The case where the tree $G$ is not homeomorphic to $D^1$. \\
$S_{1}:=\{(x_1,x_2) \in {\mathbb{R}}^2 \mid x_1=a{(x_2-p_1)}^2+p_2\}$ with $(p_1,p_2) \in {\mathbb{R}}^2$ and $a>0$.
 We can choose $l-1 \geq 1$ parabolas $S_j$ congruent to $S_1$ and parallel to this in the following way.
\begin{itemize}
\item For each $S_j$, a positive integer $2 \leq j \leq l$ is assigned. $S_{j}:=\{(x_1,x_2) \in {\mathbb{R}}^2 \mid x_1=a{(x_2-p_{1,j})}^2+p_{2,j}\}$ and $p_{1,1}:=p_1<p_{1,j}$ for each $2 \leq j \leq l$.

\item Two distinct parabolas $S_{j_1}$ and $S_{j_2}$ in $\{S_j\}_{j=1}^{l}$ intersect if and only if $j_1= j_2 \pm 1$ or $(j_1,j_2)=(1,l),(l,1)$. Furthermore, these intersections are always of the form $\{(p_{1,0},q)\}$ with a fixed real number $p_{1,0}$ satisfying $p_{1,0}>p_{1,j}$ ($1 \leq j \leq l$) and one-point sets. 
\item We have an RA-region $(D,\{S_j\}_{j=1}^l)$ such that the closure ${\overline{D}}^{{\mathbb{R}}^2}$ is compact and connected whose Poincar\'e-Reeb graph is isomorphic to $G$.
\item At each point $p_j$ in ${\overline{D}}^{{\mathbb{R}}^2} \bigcap S_j$ contained in exactly one $S_j$, a straight segment parallel to a normal vector at $p_j$ departing from ${\overline{D}}^{{\mathbb{R}}^2}$ does not contain any point of ${\overline{D}}^{{\mathbb{R}}^2}$ other than $p_j$.
\end{itemize}
See Figure \ref{fig:1} for this situation.

\begin{figure}
	\includegraphics[width=70mm,height=65mm]{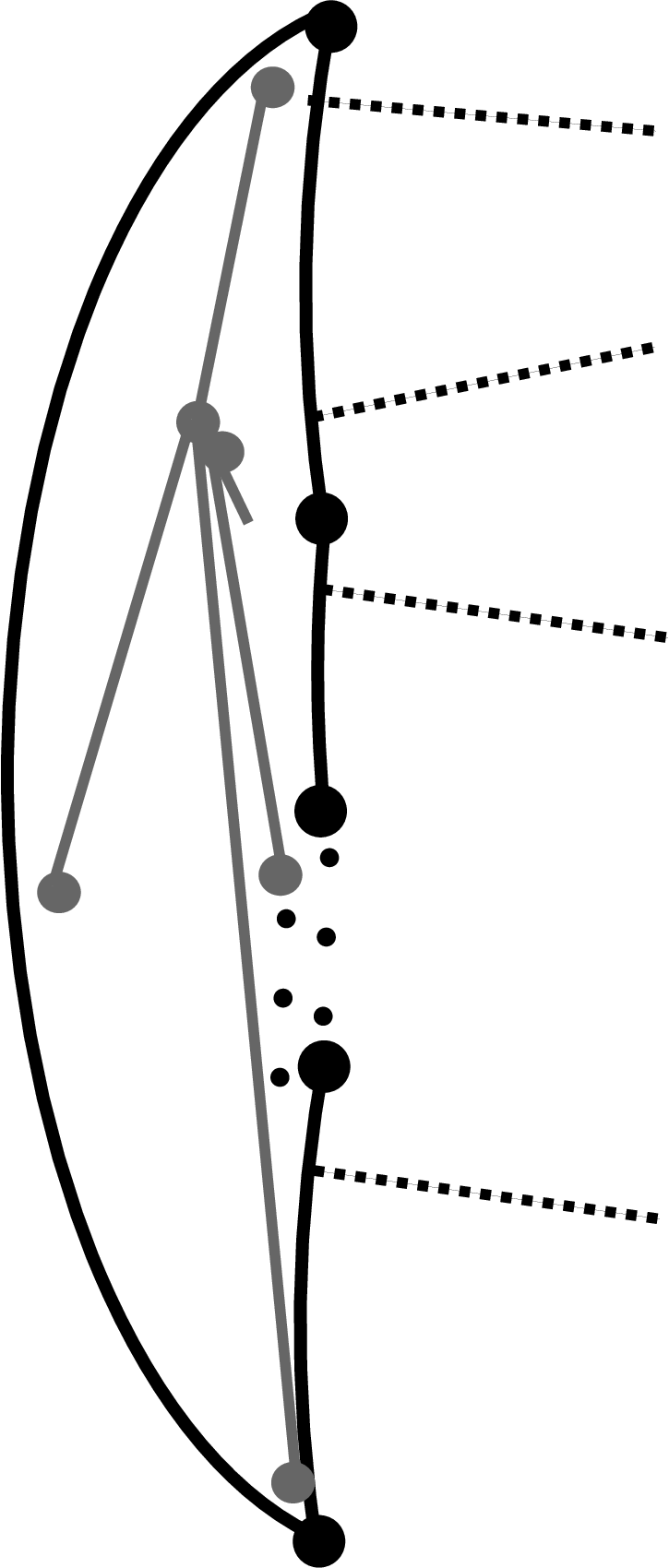}
	\caption{$(D,\{S_j\}_{j=1}^l)$ for  STEP 1-1-2. Its Poincar\'e-Reeb graph is colored in grey. Dotted straight segments show segments parallel to the normal vectors at the points. Hereafter, each "$\cdots$" in Figures is for abbreviation of certain objects.}
	\label{fig:1}
\end{figure}

For our reviewing and understanding, we add exposition on the graph structure, in STEP 1-1-2. Any connected component $I_{b,q}$ of the intersection of a line of the form $\{(b,t) \mid t \in \mathbb{R}\}$ and ${\overline{D}}^{{\mathbb{R}}^n}$ is homeomorphic to the one-point set or $D^1$.
The Poincar\'e-Reeb graph ${\rm PR}_{(D,\{S_j\}_{j=1}^{l})}$ can be defined in the following way. 

\begin{itemize}
\item Its vertex is defined as the uniquely defined value of  $q_{{\pi}_{n,1} {\mid}_{{\overline{D}}^{{\mathbb{R}}^n}}}$ on a connected component $I_{b,s}$ of the intersection of a line of the form $\{(b,t) \mid t \in \mathbb{R}\}$ and ${\overline{D}}^{{\mathbb{R}}^n}$ with either of the following.
\begin{itemize}
\item $b=p_{1,1}=p_1$.
\item $b=p_{1,0}$
\item $b=p_{1,j}$ and $I_{b,q}$ and some $S_i$ intersect in some point there in such a way that there their tangent vector spaces agree.
\end{itemize}
\item Its edge is defined as the image of a connected component of the complementary set of the union of all $I_{b,s}$ above in ${\overline{D}}^{{\mathbb{R}}^n}$ by $q_{{\pi}_{n,1} {\mid}_{{\overline{D}}^{{\mathbb{R}}^n}}}$.
\end{itemize}
\ \\
 \\
\noindent STEP 1-2 The case of a tree the degrees of whose vertices of which may be $2$. \\

We first obtain a case of STEP 1-1 and the graph homeomorphic to a given tree $G$. For the number of the parabolas here, we use $l:=l_0$, and we also use $D:=D_0$ for the region $D$. 
We consider another parabola of the form $S_{a_0,0}:=\{(x_1,x_2) \in {\mathbb{R}}^2 \mid x_1=a_0{(x_2-p_{a_0,1})}^2+p_{a_0,2}\}$ with $(p_{a_0,1},p_{a_0,2}) \in {\mathbb{R}}^2$ and a sufficiently large number $a_0>0$.

In the case of STEP 1-1, at each point $p_j$ in ${\overline{D_0}}^{{\mathbb{R}}^2} \bigcap S_j$ contained in exactly one $S_j$, a straight segment parallel to a normal vector at $p_j$ departing from ${\overline{D_0}}^{{\mathbb{R}}^2}$ does not contain any point of ${\overline{D_0}}^{{\mathbb{R}}^2}$ other than $p_j$. Due to this, we can add a parabola congruent to $S_{a_0,0}$ to add one or two vertices to each edge of the tree in STEP 1-1. To each edge of the tree, we can add vertices of an arbitrary positive number by choosing finitely many points of ${\overline{D_0}}^{{\mathbb{R}}^2}-D_0$ suitably and adding parabolas which are sufficiently close to the points and which are congruent to $S_{a,0,0}$ and mutually disjoint in ${\overline{D_0}}^{{\mathbb{R}}^2}$. More precisely, we do finitely many operations of the following.

\begin{itemize}
\item By choosing a point of the form $(p_{1,0},q)$ and put a parabola congruent to $S_{a_0,0}$ and sufficiently close to the point suitably, we can put one new vertex in the edge of the Poincar\'e-Reeb graph ${\rm PR}_{(D_0,\{S_j\}_{j=1}^{l_0})}$ incident to $q_{{\pi}_{n,1} {\mid}_{{\overline{D_0}}^{{\mathbb{R}}^n}}}(p_{1,0},q)$. See also Figure \ref{fig:2}.
\begin{figure}
	\includegraphics[width=70mm,height=65mm]{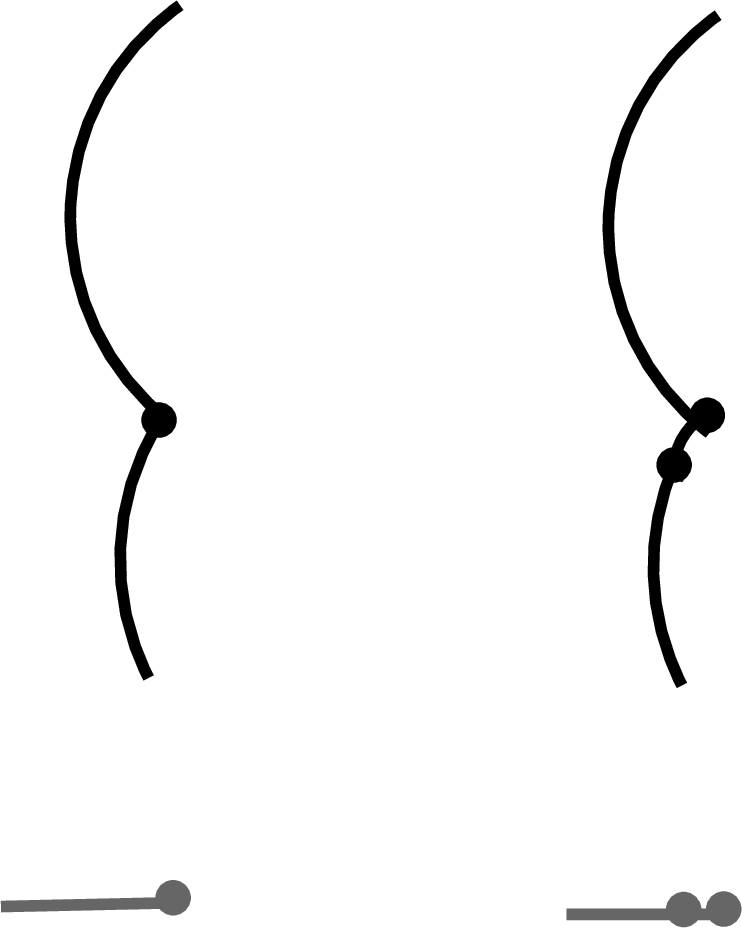}
	\caption{A point of the form $(p_{1,0},q)$ is chosen and a parabola congruent to $S_{a_0,0}$ and sufficiently close to the point is chosen suitably. The Poincar\'e-Reeb graph ${\rm PR}_{(D_0,\{S_j\}_{j=1}^{l_0})}$ is locally changed.}
	\label{fig:2}
\end{figure}
\item By choosing a point of the form $(p_{1,j},q_p) \in {\overline{D_0}}^{{\mathbb{R}}^2}-D_0$ ($2 \leq j \leq l$) with $q_p$ being an arbitrary point satisfying the condition $(p_{1,j},q_p) \in {\overline{D_0}}^{{\mathbb{R}}^2}-D_0$ and putting a parabola congruent to $S_{a_0,0}$ and sufficiently close to the point suitably, we can put one new vertex in the unique edge $e_{p_{1,j},p_{2,j}}$ of the Poincar\'e-Reeb graph ${\rm PR}_{(D_0,\{S_j\}_{j=1}^{l_0})}$, which is incident to $q_{{\pi}_{n,1} {\mid}_{{\overline{D_0}}^{{\mathbb{R}}^n}}}(p_{1,j},q)$, and the values of $\bar{{\pi}_{n,1} {\mid}_{{\overline{D_0}}^{{\mathbb{R}}^n}}}$, satisfying the relation ${\pi}_{n,1} {\mid}_{{\overline{D_0}}^{{\mathbb{R}}^n}}=\bar{{\pi}_{n,1} {\mid}_{{\overline{D_0}}^{{\mathbb{R}}^n}}} \circ q_{{\pi}_{n,1} {\mid}_{{\overline{D_0}}^{{\mathbb{R}}^n}}}$, on which, are smaller than $p_{1,j}$. See also Figure \ref{fig:3}.
\begin{figure}
	\includegraphics[width=70mm,height=65mm]{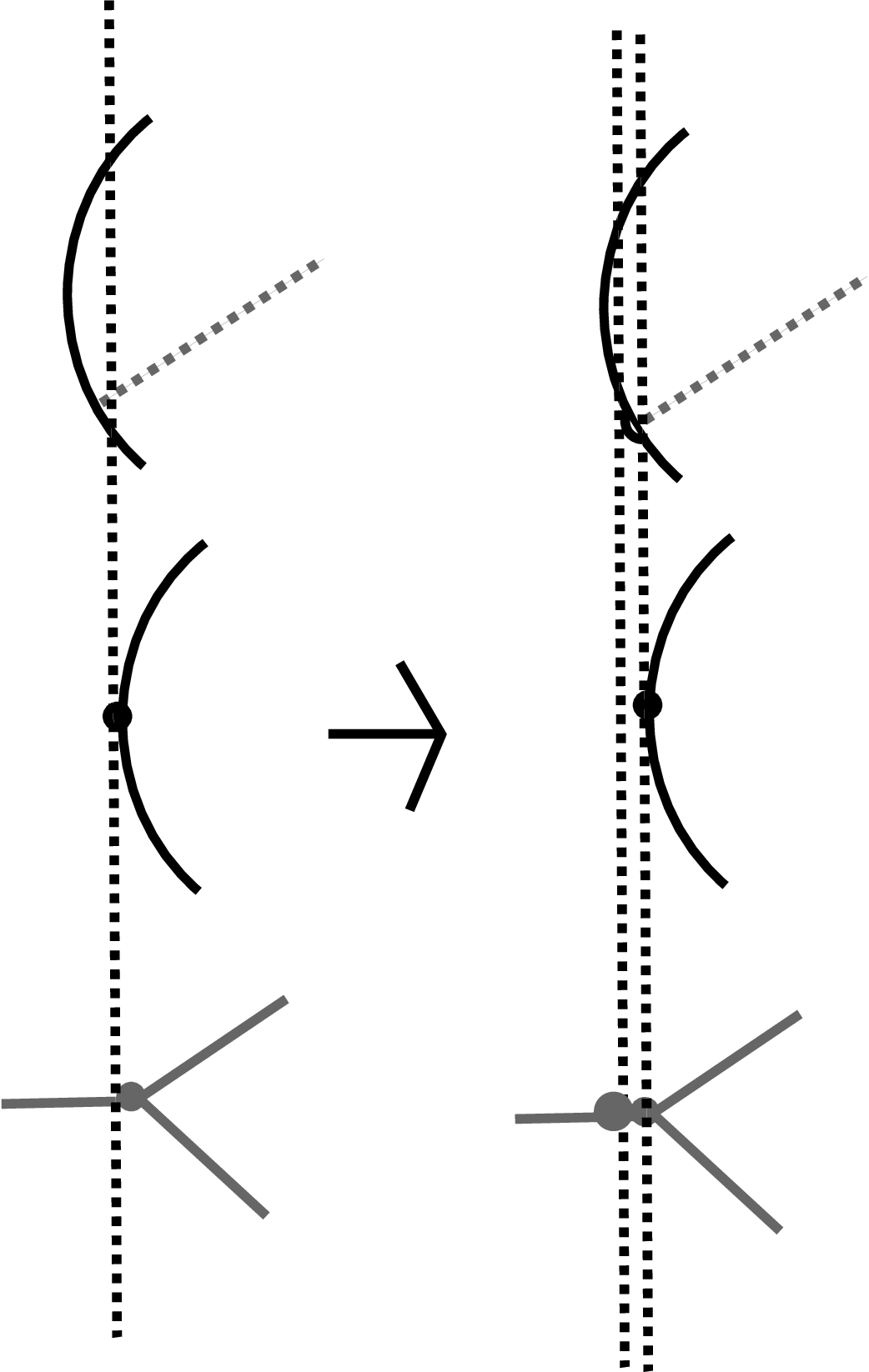}
	\caption{A point of the form $(p_{1,j},q_p) \in {\overline{D_0}}^{{\mathbb{R}}^2}-D_0$ ($2 \leq j \leq l$) is chosen and a parabola congruent to $S_{a_0,0}$ and sufficiently close to the point is chosen suitably. The Poincar\'e-Reeb graph ${\rm PR}_{(D_0,\{S_j\}_{j=1}^{l_0})}$ is locally changed.}
	\label{fig:3}
\end{figure}
\item By choosing a suitable point in ${\overline{D_0}}^{{\mathbb{R}}^2}-D_0$ except the previous points and $(p_{1,1},p_{2,1})$, and put a parabola congruent to $S_{a_0,0}$ and sufficiently close to the point suitably, we can put two new vertices in each edge of the Poincar\'e-Reeb graph ${\rm PR}_{(D_0,\{S_j\}_{j=1}^{l_0})}$. See also Figure \ref{fig:4}.
\begin{figure}
	\includegraphics[width=70mm,height=65mm]{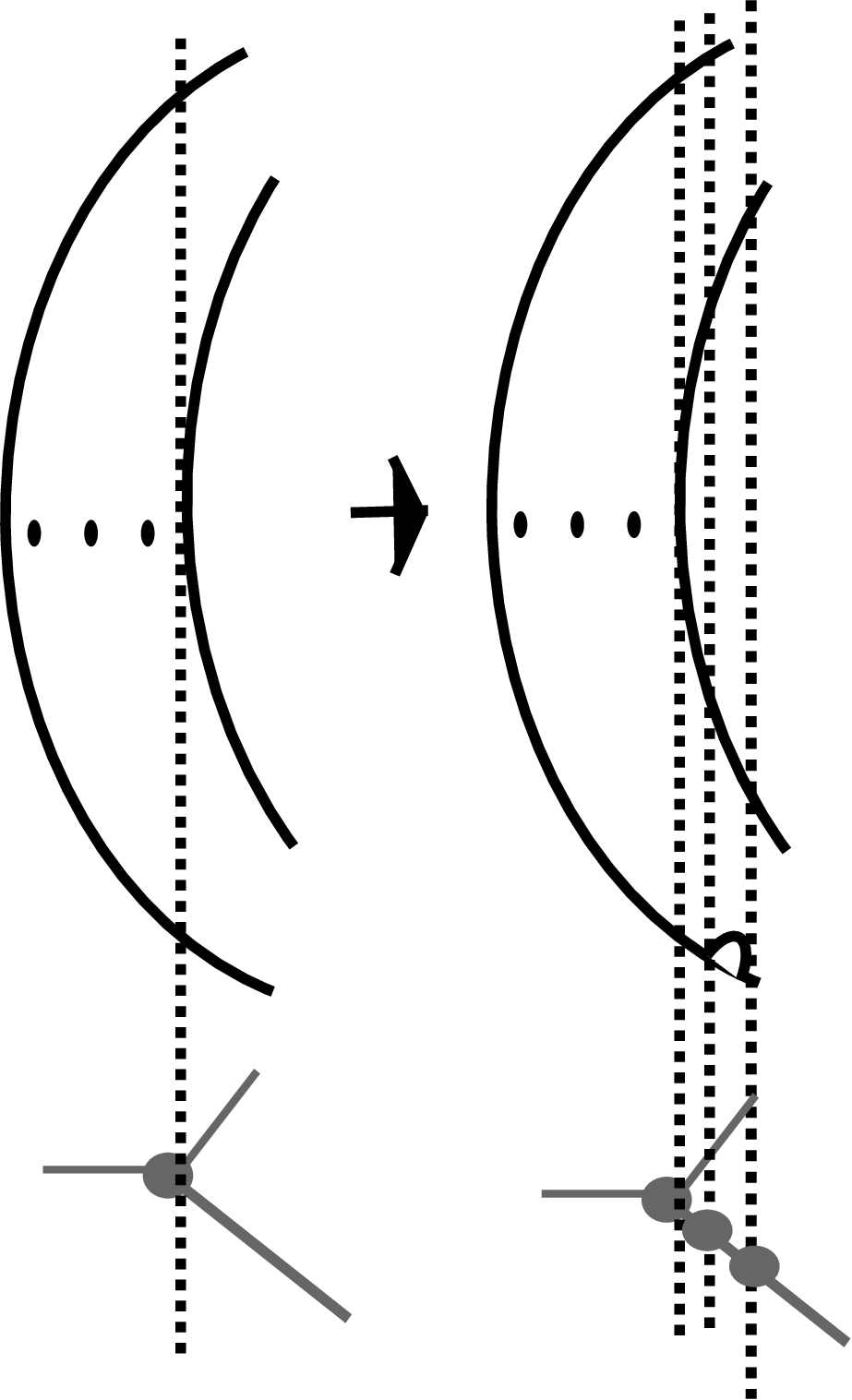}
	\caption{A point in $(S_1-(S_2 \bigcup S_l)) \bigcap {\overline{D}}^{{\mathbb{R}}^2}$ which is not of the form $(p_{1,j},q_p) \in {\overline{D}}^{{\mathbb{R}}^2}-D$ ($2 \leq j \leq l$) is chosen and a parabola congruent to $S_{a_0,0}$ and sufficiently close to the point is chosen suitably. The Poincar\'e-Reeb graph ${\rm PR}_{(D,\{S_j\}_{j=1}^{l_0})}$ is locally changed.}
	\label{fig:4}
\end{figure}
\end{itemize}

We have shown that we have arbitrary trees (up to isomorphisms) in this way.

This completes the proof.
\end{proof}
Note that by fundamental arguments and properties on real algebraic curves of degree $1$ or $2$, we can argue this in the case of circles.
\begin{Thm}
\label{thm:1}
For circles instead of parabolas, the same fact holds.
\end{Thm}
We present related history here. Note that related studies are mainly initiated by the author. 
Remember that studies on regions have been launched by the author himself, based on active studies of the author on (re)construction of real algebraic maps onto regions in real affine spaces, pioneered in the article \cite{kitazawa3}, remarked in \cite{kitazawa9} by the author himself, and the preprint \cite{kitazawa5}. Related construction is also explicitly presented in our proof of Main Theorem \ref{mthm:2}, again.

As a class we do not consider in our present paper, the class of cylinders of real algebraic manifolds in real affine spaces is important. A cylinder of a real algebraic manifold is a real algebraic manifold in ${\mathbb{R}}^{k_1+K_2}$ with $k_1$ and $k_2$ being positive integers congruent to a real algebraic manifold of the form $C \times {\mathbb{R}}^{k_2} \subset {\mathbb{R}}^{k_1} \times {\mathbb{R}}^{k_2}$, where $C$ is a real algebraic manifold in ${\mathbb{R}}^{k_1}$. 
In \cite{kitazawa6, kitazawa8}, motivated by the study \cite{kitazawa5}, regions surrounded by cylinders of circles in the plane are important. In \cite{kitazawa7}, cylinders of connected components of hyperbolas are important. Systematic studies on regions surrounded by circles are launched first in \cite{kitazawa10, kitazawa11} and important in \cite{kitazawa15, kitazawa16, kitazawa18}. \cite{kitazawa17} is also a recent preprint on a discovery of an explicit case related to cylinders of circles. \cite{kitazawa12, kitazawa19} are also preprints related to regions surrounded by real algebraic curves in the plane or RA regions and they are on certain explicit and general cases. 

We can say that cases with parabolas are first presented in \cite[Theorem 1]{kitazawa20}. Main Theorem \ref{mthm:2} is regarded as a slightly revised theorem. In \cite[Theorem 1]{kitazawa20}, the RA-region is surrounded by parabolas and two straight lines $\{(-a,t_1),(a,t_2) \mid t_1,t_2 \in \mathbb{R}\}$ with $a>0$, where the notion of RA-region for infinitely many real algebraic manifolds $S_{\lambda}$ is not defined in \cite{kitazawa20}. The construction of the map onto the closure of the region taken in ${\mathbb{R}}^2$ is slightly different from that in the proof of the original theorem.

\begin{MainThm}
\label{mthm:2}
\begin{enumerate}
\item \label{mthm:2.1}
 There exist an infinite graph $G$ and a parabola $S_P \subset {\mathbb{R}}^2$ and we have an RA-region $(D,\{S_j\}_{j \in \mathbb{Z}})$ such that $D \subset {\mathbb{R}}^2$, that $S_j$ is a parabola congruent to $S_P$, that $q_{{\pi}_{2,1} {\mid}_{{\overline{D}}^{{\mathbb{R}}^2}}}$ is proper, and that its Poincar\'e-Reeb graph ${\rm PR}_{(D,\{S_j\}_{j \in \mathbb{Z}})}$ can be defined as in Definition \ref{def:3} with Definition \ref{def:2} and is isomorphic to $G$.
\item \label{mthm:2.2}
In {\rm (}\ref{mthm:2.1}{\rm )}, we can change $S_P$ to an arbitrary parabola without changing $G$ and have the same statement.
\item \label{mthm:2.3} For each integer $m \geq 3$, we have a smooth map $e_{S_P,m}:{\mathbb{R}}^{m+2} \rightarrow {\mathbb{R}}^2$ with the following properties.
\begin{enumerate}
\item \label{mthm:2.3.1} The map is a real polynomial map on some region of the form $\{t \mid -a<t<a\} \times \mathbb{R} \times {\mathbb{R}}^m$ with a suitable positive number $a>0$ containing ${\overline{D}}^{{\mathbb{R}}^2}$.
\item \label{mthm:2.3.2} The zero set ${e_{S_P,m}}^{-1}(0)$ is an $m$-dimensional smooth submanifold of $\{t \mid -a<t<a\} \times \mathbb{R} \times {\mathbb{R}}^m$ and the restriction of ${\pi}_{m+2,2}$ there is a map onto the closure of the region $D$ in {\rm (}\ref{mthm:2.1}{\rm )} and {\rm (}\ref{mthm:2.2}{\rm )}, taken in ${\mathbb{R}}^2$.
\item \label{mthm:2.3.3} The Reeb space is regarded as the Reeb graph and isomorphic to ${\rm PR}_{(D,\{S_j\}_{j \in \mathbb{Z}})}$.
\end{enumerate} 
\end{enumerate}
\end{MainThm}
\begin{proof}

We first prove that the properties (\ref{mthm:2.1}, \ref{mthm:2.2}) are enjoyed. 
\begin{itemize}
\item Let $S_{4j}:=\{\{(x_1,x_2)\mid  x_1-{(x_2-8j)}^2+\frac{1}{2}=0\}$, where $j$ is an integer. 
\item Let $S_{4j+1}:=\{\{(x_1,x_2)\mid  x_1-{(x_2-8j-4)}^2+\frac{1}{2}=0\}$, where $j$ is an integer. 
\item Let $S_{4j+2}:=\{\{(x_1,x_2)\mid  -x_1+{(x_2-8j-2)}^2-\frac{1}{2}=0\}$, where $j$ is an integer. 
\item Let $S_{4j+3}:=\{\{(x_1,x_2)\mid  -x_1+{(x_2-8j-6)}^2-\frac{1}{2}=0\}$, where $j$ is an integer. 
\end{itemize}
They are congruent to a parabola.
For each $S_i$, by changing "$=$" in the equation to "$<$", "$>$","$\leq$" and "$\geq$" , we have $S_{i,<}$, $S_{i,>}$, $S_{i,\leq}$, and $S_{i,\geq}$, respectively.
The set $S_{4j_1+k_1} \bigcap S_{4j_2+k_2}$ is empty for any pair $(j_1,j_2)$ of integers in the case $k_1 \in \{0,1\}$ and $k_2 \in \{2,3\}$.
Let $D:={\bigcap}_{j \in \mathbb{Z}} (S_{4j,<} \bigcap S_{4j+1,<} \bigcap S_{4j+2,>} \bigcap  S_{4j+3,>})$, which is a region surrounded by these curves $S_i$.
The closure $${\overline{D}}^{{\mathbb{R}}^2}:={\bigcap}_{j \in \mathbb{Z}} (S_{4j,\leq} \bigcap S_{4j+1,\leq} \bigcap S_{4j+2,\geq} \bigcap  S_{4j+3,\geq})$$ is also important.
The set $S_i \bigcap \overline{D}$ is non-empty for any integer $i$.

The map $q_{{\pi}_{n,1} {\mid}_{{\overline{D}}^{{\mathbb{R}}^n}}}$ is proper, from our construction. We can also see that any connected component $I_{b,q}$ of the intersection of a line of the form $\{(b,t) \mid t \in \mathbb{R}\}$ and ${\overline{D}}^{{\mathbb{R}}^n}$ is homeomorphic to the one-point set or $D^1$.
The Poincar\'e-Reeb graph ${\rm PR}_{(D,\{S_j\}_{j \in \mathbb{Z}})}$ can be defined in the following way. This is same as the case where the set $\Lambda$ is finite.
\begin{itemize}
\item Its vertex is defined as the uniquely defined value of  $q_{{\pi}_{n,1} {\mid}_{{\overline{D}}^{{\mathbb{R}}^n}}}$ on a connected component $I_{b,s}$ of the intersection of a line of the form $\{(b,t) \mid t \in \mathbb{R}\}$ and ${\overline{D}}^{{\mathbb{R}}^n}$ with either of the following.
\begin{itemize}
\item $b=\pm \frac{7}{2}=\pm (2^2-\frac{1}{2})$.
\item $b=\pm \frac{1}{2}$ and $I_{b,q}$ and some $S_i$ intersect in some point in such a way that there their tangent vector spaces there agree.
\end{itemize}
\item Its edge is defined as the image of a connected component of the complementary set of the union of all $I_{b,s}$ above in ${\overline{D}}^{{\mathbb{R}}^n}$ by $q_{{\pi}_{n,1} {\mid}_{{\overline{D}}^{{\mathbb{R}}^n}}}$.
\end{itemize}

By using an affine transformation mapping $(x_1,x_2)$ to $(rx_1,x_2)$ with $r>0$, we have a similar case of parabolas congruent to any given parabola.

The properties (\ref{mthm:2.1}, \ref{mthm:2.2}) are shown to be enjoyed. 

We prove that the property (\ref{mthm:2.3}) is enjoyed for the region $D$. The closure ${\overline{D}}^{{\mathbb{R}}^2}={\overline{D}}^{{\mathbb{R}}^2}$ is also important.

We define some smooth functions. Similar functions are defined in the original proof of \cite[Theorem 2]{kitazawa20} and there exists some difference between these two situations.
\begin{itemize}
\item Let $k=0,1,2,3$. Let $c_{\mathcal{S},k}$ be a smooth function such that on ${\bigcup}_{j \in \mathbb{Z}} \{x_2 \mid 8j+2k-\frac{5}{2}<x_2<8j+2k+\frac{5}{2}\}$, real polynomial function of degree $2$ for the parabolas $S_{4j+k}$ above and that outside the set ${\bigcup}_{j \in \mathbb{Z}} \{x_2 \mid 8j+2k-\frac{5}{2}<x_2<8j+2k+\frac{5}{2}\}$ in $\mathbb{R}$, the absolute values of the values of $c_{\mathcal{S},k}$ are always greater that ${(\frac{5}{2})}^2-\frac{1}{2}=\frac{23}{4}$. 
\item Let $k=0,1,2,3$. The following are equivalent.
\begin{itemize}
\item $-\frac{7}{2} \leq c_{\mathcal{S},k}(x_2) \leq \frac{7}{2}$.

\item $8j+2k-2 \leq x_2 \leq 8j+2k+2$, $j \in \mathbb{Z}$.
\end{itemize}
\end{itemize}
We can define a new smooth map $e_{S_P,m}:{\mathbb{R}}^{m+2} \rightarrow {\mathbb{R}}^2$ defined by $e_{S_P,m}((x_1,x_2,{(y_j)}_{j=1}^{m}):=(c_{\mathcal{S},1}(x_2)-x_1) (x_1-c_{\mathcal{S},3}(x_2))-{y_1}^2,c_{\mathcal{S},2}(x_2)-x_1) (x_1-c_{\mathcal{S},4}(x_2))-{\Sigma}_{j=1}^{m-1} {y_{j+1}}^2)$. The zero set ${e_{S_P,m}}^{-1}(0)=\{(x_1,x_2,{(y_j)}_{j=1}^{m}) \in {\mathbb{R}}^{m+2} \mid (c_{\mathcal{S},1}(x_2)-x_1) (x_1-c_{\mathcal{S},3}(x_2))-{y_1}^2=0, (c_{\mathcal{S},2}(x_2)-x_1) (x_1-c_{\mathcal{S},4}(x_2))-{\Sigma}_{j=1}^{m-1} {y_{j+1}}^2=0\}$ of the map $e_{S_P,m}$ is also represented as ${e_{S_P,m}}^{-1}(0)=\{(x_1,x_2,{(y_j)}_{j=1}^{m}) \in {\overline{D}}^{{\mathbb{R}}^2} \times {\mathbb{R}}^{m} \mid e_{S_P,m}(x_1,x_2,{(y_j)}_{j=1}^{m})=0\}$.

We prove that the properties (\ref{mthm:2.3.1}, \ref{mthm:2.3.2}) are enjoyed with $\frac{7}{2}<a \leq \frac{23}{4}$. The condition $\frac{7}{2}<a \leq \frac{21}{4}$ makes the map $e_{S_P,m}$ a real polynomial map on the region $\{t \mid -a<t<a\} \times \mathbb{R} \times {\mathbb{R}}^m$, containing ${\overline{D}}^{{\mathbb{R}}^2}$ in its interior considered in ${\mathbb{R}}^2$.
We prove that ${e_{S_P,m}}^{-1}(0)=\{(x_1,x_2,{(y_j)}_{j=1}^{m}) \in {\overline{D}}^{{\mathbb{R}}^2} \times {\mathbb{R}}^{m} \mid e_{S_P,m}(x_1,x_2,{(y_j)}_{j=1}^{m})=0\}$
is an $m$-dimensional smooth submanifold of $\{t \mid -a<t<a\} \times \mathbb{R} \times {\mathbb{R}}^m$ by implicit function theorem, as in \cite{kitazawa20}. Related methods are originally presented first in \cite{kitazawa5}. 
We do not assume related knowledge or arguments.
We can easily see that the restriction of ${\pi}_{m+2,2}$ there is a map onto the ${\overline{D}}^{{\mathbb{R}}^2}$. \\

Hereafter, let the curve $\{(c_{\mathcal{S},k}(x_2),x_2) \mid x_2 \in \mathbb{R}\}$ in ${\mathbb{R}}^2$ be denoted by $S_{c_{\mathcal{S},k}}$. For $k_1 \in \{1,2\}$ and $k_2 \in \{3,4\}$, the intersection of  $S_{c_{\mathcal{S},k_1}}$ and $S_{c_{\mathcal{S},k_2}}$ is empty.\\
\ \\
Case 2-1 At the point $(x_1,x_2,{(y_j)}_{j=1}^{m}) \in {e_{S_P,m}}^{-1}(0)$ with $(x_1,x_2) \in D$. \\
The value of the partial derivative of the real polynomial function $(c_{\mathcal{S},1}(x_2)-x_1) (x_1-c_{\mathcal{S},3}(x_2))-{y_1}^2$ by $y_j$ with $j \geq 2$ ($y_1$) is (resp. not) $0$, there. 
The value of the partial derivative of the real polynomial function $(c_{\mathcal{S},2}(x_2)-x_1) (x_1-c_{\mathcal{S},4}(x_2))-{\Sigma}_{j=1}^{m-1} {y_{j+1}}^2$ by $y_1$ (resp. $y_j$ with some $j \geq 2$) is (resp. not) $0$, there.
The rank of the map $e_{S_P,m}$ there is $2$. \\
\ \\
Case 2-2 At the point $(x_1,x_2,{(y_j)}_{j=1}^{m}) \in {e_{S_P,m}}^{-1}(0)$ with $(x_1,x_2) \in {\overline{D}}^{{\mathbb{R}}^2}-D$ and the value of the exactly one of two real polynomial functions there is $0$. \\
The value of the partial derivative of the real polynomial function $(c_{\mathcal{S},1}(x_2)-x_1) (x_1-c_{\mathcal{S},3}(x_2))-{y_1}^2$ by each $y_j$ is $0$, there, in the case the value of the real polynomial function $(c_{\mathcal{S},1}(x_2)-x_1) (x_1-c_{\mathcal{S},3}(x_2))-{y_1}^2$ is $0$ there. In this case, at the point, the value of the partial derivative of the real polynomial function $(c_{\mathcal{S},1}(x_2)-x_1) (x_1-c_{\mathcal{S},3}(x_2))-{y_1}^2$ by some $x_j$ is not $0$ by the smoothness of the curve $S_{c_{\mathcal{S},k}}$ with $k=1,3$. At the point, the value of the partial derivative of the real polynomial function $(c_{\mathcal{S},2}(x_2)-x_1) (x_1-c_{\mathcal{S},4}(x_2))-{\Sigma}_{j=1}^{m-1} {y_{j+1}}^2$ by $y_j$ with some $j \geq 2$ is not $0$. The rank of the map $e_{S_P,m}$ at the point is $2$.

By the symmetry, we can prove this in the case the value of the real polynomial function $(c_{\mathcal{S},2}(x_2)-x_1) (x_1-c_{\mathcal{S},4}(x_2))-{\Sigma}_{j=1}^{m-1} {y_{j+1}}^2$ at the point is $0$.\\
.\ \\
Case 2-3 At the point $(x_1,x_2,{(y_j)}_{j=1}^{m}) \in {e_{S_P,m}}^{-1}(0)$ with $(x_1,x_2) \in {\overline{D}}^{{\mathbb{R}}^2}-D$ and the values of these two real polynomial functions there are both $0$. \\
The values of the partial derivatives of the two real polynomial functions by each $y_j$ are $0$, there. The value of the partial derivative of each of these real polynomial functions by some $x_j$ is not $0$ by the smoothness of the curves $S_{c_{\mathcal{S},k}}$. Due to the location of the parabolas, at the point, the curves $S_{c_{\mathcal{S},k_1}}$ and $S_{c_{\mathcal{S},k_2}}$ intersect in such a way that the two normal vectors chosen for these two curves are mutually independent. The rank of the map $e_{S_P,m}$ at the point is $2$. \\
\ \\
By implicit function theorem, the properties (\ref{mthm:2.3.1}, \ref{mthm:2.3.2}) are shown to be enjoyed. \\
We discuss the property (\ref{mthm:2.3.3}). By the theory \cite{gelbukh, saeki1, saeki2} with the structure of the manifolds and maps, it is immediately shown that the property (\ref{mthm:2.3.3}) is enjoyed. \\

This completes the proof.
\end{proof}
In Main Theorem \ref{mthm:1}, as presented in Main Theorem \ref{mthm:2} (\ref{mthm:2.3}), we can construct smooth maps and functions in the real algebraic situation. This is a kind of exercises. For this, consult also the preprint \cite{kitazawa5, kitazawa20}, arguments related to which we do not assume, for example.
 \section{Conflict of interest and Data availability.}
  \noindent {\bf Conflict of interest.} \\
 The author is a researcher at Osaka Central Advanced Mathematical Institute (OCAMI researcher). The institute is supported by MEXT Promotion of Distinctive Joint Research Center Program JPMXP0723833165. He is not employed there. He thanks this. \\
  %Some of works by other researchers and this version may overlap in some of the contents due to the nature that our problems are natural in theory of Morse functions and applications to differential topology and that related mathematical studies are very fundamental and classical in some senses, for example. However the present version of our paper is presented independent of these work. \\
  %Saga Souhatsu Mathematical Seminar (http://inasa.ms.saga-u.ac.jp/Japanese/saga-souhatsu.html), inviting the author as a speaker, is funded and supported by JST Fusion Oriented REsearch for disruptive Science and Technology JPMJFR202U: the author was a speaker on 2024/7/12 supported by this project.\\
  \ \\
  {\bf Data availability.} \\
 No data other than the present file is generated, essentially. Non-trivial arguments from preprints (of the author) are not assumed in arguments in the present paper.

\end{document}